\documentclass[reqno]{amsart}
\usepackage[utf8x]{inputenc}  

\usepackage{setspace}
\usepackage[left=3.1cm, right=3.1cm, bottom=4cm]{geometry}                 

\usepackage{graphicx} 
\usepackage{pgf,tikz}
\usetikzlibrary{arrows}
\usepackage{amssymb}
\usepackage{pdfsync}
\usepackage{mathrsfs}
\usepackage{hyperref} 

\usepackage{epstopdf}
\usepackage{bbm} 
\usepackage[colorinlistoftodos,prependcaption,textsize=tiny]{todonotes}
\usepackage{xargs}
\usepackage{bm}
\usepackage{lmodern}

\def\R{\mathbb{R}}

\def\Z{\mathbb{Z}}

\def\F{\mathbb{F}}

\renewcommand{\d}{\text{\rm d}}

\def\sgn {{\rm sgn}}

\newtheorem{theorem}{Theorem}

\newtheorem{question}[theorem]{Question}

\newtheorem*{definition*}{Definition}

\newtheorem{lemma}[theorem]{Lemma}

\makeatletter
\DeclareFontFamily{U}{tipa}{}
\DeclareFontShape{U}{tipa}{m}{n}{<->tipa10}{}
\newcommand{\arc@char}{{\usefont{U}{tipa}{m}{n}\symbol{62}}}%

\makeatother

\numberwithin{equation}{section}

\allowdisplaybreaks

\newcommand{\intav}[1]{\mathchoice {\mathop{\vrule width 6pt height 3 pt depth  -2.5pt
\kern -8pt \intop}\nolimits_{\kern -6pt#1}} {\mathop{\vrule width
5pt height 3  pt depth -2.6pt \kern -6pt \intop}\nolimits_{#1}}
{\mathop{\vrule width 5pt height 3 pt depth -2.6pt \kern -6pt
\intop}\nolimits_{#1}} {\mathop{\vrule width 5pt height 3 pt depth
-2.6pt \kern -6pt \intop}\nolimits_{#1}}}

\newcommand{\intavl}[1]{\mathchoice {\mathop{\vrule width 6pt height 3 pt depth  -2.5pt
\kern -8pt \intop}\limits_{\kern -6pt#1}} {\mathop{\vrule width 5pt
height 3  pt depth -2.6pt \kern -6pt \intop}\nolimits_{#1}}
{\mathop{\vrule width 5pt height 3 pt depth -2.6pt \kern -6pt
\intop}\nolimits_{#1}} {\mathop{\vrule width 5pt height 3 pt depth
-2.6pt \kern -6pt \intop}\nolimits_{#1}}}

\title[Littlewood's estimates]{Littlewood's estimates for $L$-functions \\ in the hyperelliptic ensemble}

\author[Carneiro]{Emanuel Carneiro}
\address[Emanuel Carneiro]{ICTP - The Abdus Salam International Centre for Theoretical Physics, 
Strada Costiera, 11, I - 34151, Trieste, Italy.}
\email{carneiro@ictp.it}

\author[Darbar]{Pranendu Darbar}
\address[Pranendu Darbar]{School of Mathematics and Statistics, University of New South Wales, Sydney NSW 2052, Australia.}
\email{darbarpranendu100@gmail.com}

\author[Das]{Mithun Kumar Das}
\address[Mithun Kumar Das]{ICTP - The Abdus Salam International Centre for Theoretical Physics, Strada Costiera, 11, I - 34151, Trieste, Italy; and School of Mathematical Sciences, National Institute of Science Education and Research, A CI of Homi Bhabha National Institute, Jatni, Khurda, 752050, India.}
\email{mdas@ictp.it}

\author[Ismoilov]{Tolibjon Ismoilov}
\address[Tolibjon Ismoilov]{SISSA - Scuola Internazionale Superiore di Studi Avanzati, Via Bonomea 265, 34136 Trieste, Italy.}
\email{tolibjon.ismoilov@sissa.it}

\author[Ramos]{Antonio Pedro Ramos}
\address[Antonio Pedro Ramos]{IMPA - Instituto de Matem\'{a}tica Pura e Aplicada, Estrada Dona Castorina, 110, CEP 22460-320, Rio de Janeiro, RJ - Brazil.}
\email{antonio.ramos@impa.br}

\date{\today}                                           
\begin{document}

\subjclass[2020]{11M06, 11M26, 11G20, 42A10}
\keywords{Riemann zeta-function, argument function, quadratic $L$-functions, function fields, hyperelliptic ensemble, extremal trigonometric polynomials.}
\begin{abstract} We investigate the analogues of certain classical estimates of Littlewood for the Riemann zeta-function in the context of quadratic Dirichlet $L$-functions over function fields. In some situations, we are actually able to establish finer results in the function field setup than what is currently known in the original number field setup, and this leads us to an educated guess on what could happen for the Riemann zeta-function in such situations. Fourier analysis techniques play an important role in our approach. 
\end{abstract}

\maketitle

\section{Introduction}

\subsection{Background} Let $\zeta(s)$ denote the Riemann zeta-function and, as usual, let $N(t)$ be the number of non-trivial zeros $\rho = \beta + i \gamma$ of $\zeta(s)$ with ordinates $\gamma$ in the interval $(0,t]$ (zeros with ordinate $t$ are counted with weight $\tfrac12$). For $t \geq 2$, one has
$$N(t) = \frac{t}{2\pi}\log \frac{t}{2\pi} - \frac{t}{2\pi} + \frac{7}{8} + S(t) + O\left(\frac{1}{t}\right).$$
The argument function $S(t)$ that appears in this classical formula is defined by $S(t) := \frac{1}{\pi} \arg \zeta(\tfrac12 + it)$, if $t$ is not the ordinate of a zero of $\zeta(s)$, where the argument is obtained by continuous variation along the line segments joining the points $2$, $2 + it$, and $\tfrac 12 + it$, with the convention that $\arg \zeta(2) = 0$. If $t$ is the ordinate of a zero of $\zeta(s)$, we set $S(t) := \tfrac12 \lim_{\varepsilon \to 0} \big(S(t - \varepsilon) + S(t + \varepsilon)\big)$. Useful information on the oscillatory behavior of $S(t)$ is encoded in its antiderivatives. It is convenient to consider a specific sequence of antiderivatives by setting $S_0(t) = S(t)$ and, for $n\geq 1$ and $t>0$, 
\begin{equation*}
S_n(t) := \int_0^t S_{n-1}(\tau) \,\d \tau + \delta_n\,,
\end{equation*} 
where, at each step, the constant $\delta_n$ is suitably chosen in alignment with the expected cancellation. These are explicitly given by $\delta_1 = 0$, $\delta_{2k} = \frac{(-1)^{k-1}}{(2k)! \cdot 2^{2k}}$, and  
$$\delta_{2k-1} =\frac{ (-1)^{k-1}}{\pi} \int_{\tfrac{1}{2}}^{\infty} \int_{\sigma_{2k-2}}^{\infty} \ldots \int_{\sigma_{2}}^{\infty} \int_{\sigma_{1}}^{\infty} \log |\zeta(\sigma_0)|\, \d\sigma_0\,\d\sigma_1\,\ldots \,\d \sigma_{2k-2} $$
for $k \geq 2$. For more details on these particular constants, see \cite{CC, F1, L}.

\smallskip

Assuming the Riemann hypothesis (RH),  Littlewood \cite{L} in 1924 proved that
\begin{equation}\label{20250903_14:42}
\log \big|\zeta(\tfrac12 + it)\big| \ll \frac{\log t}{\log \log t} \ \ {\rm and} \ \ |S_n(t)| \ll \frac{\log t}{(\log \log t)^{n+1}}.
\end{equation}
The order of magnitude of such upper bounds has never been improved over the past century, and the efforts have hence been concentrated in optimizing the values of the implicit constants. The current best versions of such estimates (all for sufficiently large $t$) are
\begin{equation}\label{20250903_14:28}
\log \big|\zeta(\tfrac12 + it)\big| \leq \left( \frac{\log 2}{2} + o(1) \right)  \frac{\log t}{\log \log t}\,,
\end{equation} 
due to Chandee and Soundararajan \cite{CS},
\begin{equation}\label{20250903_14:29}
|S(t)| \leq \left( \frac{1}{4} + o(1) \right)  \frac{\log t}{\log \log t}\,,
\end{equation} 
due to Carneiro, Chandee and Milinovich \cite{CCM, CCM2} and, more generally, for $n \geq 1$, 
\begin{equation}\label{20250903_14:30}
-\left( C^-_n + o(1)\right) \frac{\log t}{(\log \log t)^{n+1}} \ \leq \ S_n(t) \ \leq \ \left( C^+_n + o(1)\right) \frac{\log t}{(\log \log t)^{n+1}}\,,
\end{equation}
where, for $n = 4k \pm 1$,
\begin{equation}\label{20250911_17:58}
C_{n}^{\mp} = \frac{\zeta(n+1)}{\pi \cdot 2^{n+1}} \ \ \ \ {\rm and} \ \ \ \ C_{n}^{\pm} = \frac{\left( 1 - 2^{-n}\right)\zeta(n+1)}{\pi \cdot 2^{n+1}}\,,
\end{equation}
and, for $n \geq 2$ even,
\begin{align}\label{20250911_23:59}
\begin{split}
\!\!\!\!\!\!C_n^+ \!= C_n^-  \!=  \! \left[\frac{2 \big(C_{n+1}^{+} + C_{n+1}^{-}\big)  C_{n-1}^{+} C_{n-1}^{-}}{C_{n-1}^{+} + C_{n-1}^{-}}\right]^{1/2} \! \!\!\! \! \!= \frac{\sqrt{2}}{\pi  \, 2^{n+1}}\!\! \left[\frac{\!\left(1 \!- \!2^{-n-2}\right)\!\left( 1 \!- \! 2^{-n+1}\right)\zeta(n)  \zeta(n+2)}{\left(1 \!- \!2^{-n}\right)}\right]^{1/2}\!\!\!\!\!\!.
\end{split}
\end{align}
The case $n=1$ of \eqref{20250903_14:30} is due to Carneiro, Chandee and Milinovich \cite{CCM}, while the cases $n\geq 2$ in \eqref{20250903_14:30} are due to Carneiro and Chirre \cite{CC}. 

\smallskip

Estimates \eqref{20250903_14:28}, \eqref{20250903_14:29} and \eqref{20250903_14:30} are the ones we want to discuss in a function field setup. 

\smallskip

\noindent {\sc Remark}: Observe that, in light of Littlewood's bounds \eqref{20250903_14:42} under RH, the choice of constants $\delta_n$ is the only one that makes $S_n(t)$ have mean-value zero for all $n\geq 0$, that is 
\begin{equation}\label{20250905_14:51}
\lim_{t\to \infty} \frac{1}{t} \int_0^t S_n(\tau) \, \d \tau = 0.
\end{equation}

\subsection{Quadratic $L$-functions in the hyperelliptic ensemble} We now move our discussion to the setup of $L$-functions over function fields, a topic of interest in modern research in number theory; see, e.g., \cite{AGK, BF, BF2, FR, Flo1, R1}. A general reference for the material in \S \ref{Setup_sub} is \cite{Ros}.

\subsubsection{Setup} \label{Setup_sub} Start by fixing a finite field $\mathbb{F}_{q}$ of odd cardinality, and let $\mathbb{F}_{q}[x]$ be the polynomial ring over $\mathbb{F}_{q}$ in the variable $x$. For $f  \in \mathbb{F}_{q}[x]$ its degree is denoted by $\deg(f)$. If $f$ is a non-zero polynomial in $\mathbb{F}_{q}[x]$ then the norm of $f$ is defined as $|f|:=q^{\deg(f)}$. If $f=0$, we set $|f|=0$. A monic irreducible polynomial $P$, of degree at least $1$, is called a {\it prime polynomial} or simply a {\it prime}. If $P$ is a prime polynomial, the quadratic residue symbol $\big( \frac{f}{P}\big)$ is defined by
\begin{equation*}
\left( \frac{f}{P}\right) = 
\begin{cases}
0, & {\rm if} \ P\mid f;\\
1, & {\rm if} \ P \nmid f \ {\rm and} \ f \ {\rm is \ a \ square \ modulo}\ P;\\
-1, & {\rm if} \ P \nmid f \ {\rm and} \ f \ {\rm is \ not \ a \ square \ modulo}\ P.
\end{cases}
\end{equation*}
If $D \in \mathbb{F}_{q}[x]$ has prime factorization $D = P_1^{a_1}\cdots P_k^{a_k}$, we set $\left( \frac{f}{D}\right) := \prod_{i=1}^k \left( \frac{f}{P_i}\right)^{a_i}$. The quadratic reciprocity law states that for $A, B$ non-zero, coprime monic polynomials, 
\begin{equation*}
\left( \frac{A}{B}\right) = \left( \frac{B}{A}\right) (-1)^{\frac{(q-1)}{2} \deg (A) \deg(B)}.
\end{equation*}

For $d = 2g +1$ an odd natural number, we define the {\it hyperelltiptic ensemble} $\mathcal{H}_{d}$ as
\begin{equation}\label{20250908_16:48}
\mathcal{H}_d=\left\{ D\in \mathbb{F}_{q}[x] : \, D \text{ is monic, square-free, and} \, \deg(D)=d \right\}.
\end{equation}
For each $D \in \mathcal{H}_{d}$, one can consider the quadratic Dirichlet character $\chi_D$ defined by $\chi_D(f) = \big( \frac{D}{f}\big)$, and the corresponding $L$-function 
\begin{align*}
L(s,\chi_{D})=\sum_{f \ {\rm monic}} \frac{\chi_{D}(f)}{|f|^{s}}=\prod_{P \ {\rm prime}}\left(1-\chi_{D}(P)\,|P|^{-s} \right)^{-1},\;\; \operatorname{Re}(s)>1.
\end{align*} 
It is convenient in this setup to sometimes work with an equivalent function written in terms of the variable  $u=q^{-s}$, namely,
\begin{align}\label{20250908_16:49}
\mathcal{L}(u,\chi_{D})=\sum_{f\ {\rm monic}} \chi_{D}(f)\, u^{\deg(f)}=\prod_{P\ {\rm prime}}\left(1-\chi_{D}(P)\,u^{\deg(P)} \right)^{-1},\;\;\, |u|<\frac{1}{q} .
\end{align}

A crucial fact here is that the function $\mathcal{L}(u,\chi_{D})$ turns out to be  a polynomial of degree $2g$; see \cite[Proposition 4.3]{Ros}. Moreover, from the classical work of Weil  \cite{WEIL} on the Riemann hypothesis for curves over finite fields (which, in our context, boils down to the fact that for each $D \in \mathcal{H}_{d}$ there is an associated hyperelliptic curve $C_D : y^2 = D(x)$, non-singular and of genus $g$), it follows that $\mathcal{L}(u,\chi_{D})$ satisfies the functional equation 
\begin{align}\label{20250910_15:16}
\mathcal{L}(u,\chi_{D}) = (qu^2)^g \,\mathcal{L}\left(\frac{1}{qu},\chi_{D}\right),
\end{align}
and verifies the Riemann hypothesis, namely, all of its zeros lie in the circle $|u| = q^{-1/2}$; see, e.g., \cite[\S 2]{AGK}. Letting $e(x) := e^{2 \pi i x}$, we then write
\begin{align}\label{20250908_17:40}
\mathcal{L}(u,\chi_{D})= \prod_{j=1}^{2g} \left(1 -  \frac{u}{u_j}\right)\,,
\end{align}
where $u_j := q^{-1/2} e(\theta_j)$ are the zeros of $\mathcal{L}(u,\chi_{D})$. From the functional equation, note that the complex zeros appear in conjugate pairs, and hence $\mathcal{L}(u,\chi_{D})$ is real-valued for $u \in \R$.

\smallskip

We want to look into analogues of \eqref{20250903_14:28}, \eqref{20250903_14:29} and \eqref{20250903_14:30} as $d \to \infty$. With the change of variables, the critical line $s = \tfrac12 + it$ becomes the critical circle $u = q^{-1/2} e(\theta)$, for $\theta \in \R/\Z$. Note that, since the Riemman hypothesis is true for $\mathcal{L}(u,\chi_{D})$, the estimates below, differently than  \eqref{20250903_14:28}, \eqref{20250903_14:29} and \eqref{20250903_14:30}, are unconditional.

\subsubsection{The modulus on the critical circle} The analogue of \eqref{20250903_14:28} is the following estimate.
\begin{theorem} \label{Thm1}
For  $d = 2g +1$, $D \in \mathcal{H}_d$ and $\theta  \in \R/\Z$, we have, as $d \to \infty$,
\begin{align}\label{20250904_16:17}
\log\big|\mathcal{L}\big(q^{-1/2} e(\theta),\chi_{D}\big)\big| \leq \left( \frac{\log 2}{2} + o(1) \right)  \frac{d}{\log_q d}.
\end{align}
\end{theorem}
This result is not new. It was independently obtained by Florea \cite[Corollary 8.2]{Flo1} and by Bucur, Costa, David, Guerreiro and Lowry-Duda \cite[Theorem 20]{BCDGL}, with different proofs. For the convenience of the reader we provide a short proof within our framework (that is similar in spirit to the proof given in \cite[Theorem 20]{BCDGL}, only slightly different). We note that a previous result of Altu\u{g} and Tsimerman \cite{AS} had \eqref{20250904_16:17} with the constant $(\log 2)/ 2$ replaced by $1$.

\subsubsection{The argument on the critical circle} Following the notation of Andrade, Gonek and Keating \cite[\S 4]{AGK}, when $q^{-1/2}e(\theta)$ is not a zero of $\mathcal{L}(u,\chi_{D})$ we define 
\begin{align}\label{20250911_14:12}
S(\theta, \chi_{D}) := - \frac{1}{\pi} \, \Delta_{\Gamma_{\theta}}\arg \mathcal{L}(u,\chi_{D})\,,
\end{align}
where the path $\Gamma_{\theta}$ consists of the positively oriented circular arc from $q^{-2}$ to $q^{-2}e(\theta)$ followed by the radial segment from $q^{-2}e(\theta)$ to $q^{-1/2}e(\theta)$, and $\Delta_{\Gamma_{\theta}}\arg$ denotes the variation of the argument along this path. If $q^{-1/2}e(\theta)$ is a zero of $\mathcal{L}(u,\chi_{D})$, set $S(\theta, \chi_{D}) := \tfrac12 \lim_{\varepsilon \to 0} \big(S(\theta - \varepsilon, \chi_{D}) + S(\theta + \varepsilon, \chi_{D}) \big)$. By the argument principle, note that this is indeed a $1$-periodic function on $\theta$. The estimate 
\begin{align*}
S(\theta, \chi_{D}) \ll \frac{d}{\log_q d}
\end{align*}
was established in \cite[Theorem 9]{AGK}, and a careful reading of their proof yields 
\begin{align}\label{20250911_14:20}
S(\theta, \chi_{D}) \leq \left( \frac{1}{2} + o(1) \right) \frac{d}{\log_q d}.
\end{align}
We refine this estimate to obtain the analogue of \eqref{20250903_14:29}.
\begin{theorem} \label{Thm2}
For  $d = 2g +1$, $D \in \mathcal{H}_d$ and $\theta  \in \R/\Z$, we have, as $d \to \infty$,
\begin{align*}
S(\theta, \chi_{D}) \leq \left( \frac{1}{4} + o(1) \right) \frac{d}{\log_q d}.
\end{align*}
\end{theorem}
We shall give two different proofs of this result.

\subsubsection{The antiderivatives of the argument} Note that $\theta \mapsto S(\theta, \chi_{D})$ has mean value zero in $\R/\Z$; see Lemma \ref{Lem:Bernoulli_rep} below. It is then convenient to consider a specific sequence of antiderivatives by setting $S_0(\theta, \chi_{D}) = S(\theta, \chi_{D})$ and, for $n \geq 1$ and $\theta \in \R/\Z$,
\begin{align}\label{20250908_17:29}
S_n(\theta, \chi_{D}) := \int_0^{\theta} S_{n-1}(\alpha, \chi_{D}) \,\d \alpha + c_n\,,
\end{align}
where, at each step, the constant $c_n = c_n (\chi_{D})$ is chosen in a way so that $\theta \mapsto S_n(\theta, \chi_{D})$ has mean-value zero in $\R/\Z$, i.e., $ \int_{\R/\Z} S_{n}(\theta, \chi_{D}) \,\d \theta = 0$, in sympathy with the case of the Riemann zeta-function, as described in \eqref{20250905_14:51}. We give an explicit expression for $c_n$ in \eqref{20250916_12:01} below. Observe that $S_0$ has jump discontinuities at the $\theta_j$'s, but all the functions $S_n$, for $n\geq 1$, are continuous. Note that $S_{n-1}$ with mean value zero in $\R/\Z$ ensures that $S_n$ is indeed a $1$-periodic function.

\smallskip

Before we state our next result, we need to briefly recall some basic facts about the Bernoulli polynomials and the Bernoulli periodic functions. These functions will play a relevant role in this work. The Bernoulli polynomials $B_n(x)$ can be defined by the power series expansion
\begin{equation*}
\dfrac{t \,e^{xt}}{e^t -1} = \sum_{n=0}^{\infty}\dfrac{B_n(x)}{n!}t^n\,,
\end{equation*}
where $|t| < 2\pi$. The first Bernoulli polynomials are $B_0(x) =1$, $B_1(x) = x -\tfrac12$, $B_2(x) = x^2 - x + \tfrac16$, $B_3(x)= x^3 - \tfrac32x^2 + \tfrac12x$. For $n \geq 0$, $n \neq 1$, let us define the Bernoulli periodic functions $\mathcal{B}_n(x)$ by
\begin{equation}\label{20250908_16:27}
\mathcal{B}_n(x) := B_n(x - \lfloor x \rfloor)\,,
\end{equation}
where $\lfloor x \rfloor$ denotes the largest integer not exceeding $x$. For $n = 1$ we make a minor adjustment to include a renormalization at $x = 0$ in \eqref{20250908_16:27}, and define 
\begin{align*}
\mathcal{B}_1(x) :=
\begin{cases}
 x - \lfloor x \rfloor - \frac12 & \textrm{if} \ x\notin \Z;\\
 0 & \textrm{if} \ x\in \Z.
 \end{cases}
\end{align*}
 In this way, the function $\mathcal{B}_1:\R/\Z \to \R$ is the classical sawtooth function in this paper. Throughout the rest of the paper we let
\begin{align}\label{20250912_17:34}
M_n:= \max_{x \in [0,1]} B_n(x) \ \ \ \ {\rm and} \ \ \ \ m_n:= \min_{x \in [0,1]} B_n(x).
\end{align}
 
\smallskip

The next result is not only the analogue of \eqref{20250903_14:30}, but a refined analogue of it in half of the cases, as we explain in \S \ref{Sec1.3.1} below. 
\begin{theorem} \label{Thm3}
For  $d = 2g +1$, $D \in \mathcal{H}_d$, $n \geq1 $ and $\theta  \in \R/\Z$, we have, as $d \to \infty$,
\begin{align}\label{20250912_13:18}
- \left(\frac{A_n^-}{(2\pi)^n}  + o(1) \right)  \frac{d}{(\log_q d)^{n+1}} \leq S_n(\theta, \chi_{D}) \leq \left(\frac{A_n^+}{(2\pi)^n}   + o(1) \right) \frac{d}{(\log_q d)^{n+1}}\,,
\end{align}
where 
\begin{align}\label{20250912_00:19}
A_n^- = \frac{\pi^n}{2} \cdot \frac{M_{n+1}}{(n+1)!}  \ \ \ {\rm and} \ \ \ A_n^+ =- \frac{\pi^n}{2}\cdot \frac{m_{n+1}}{(n+1)!}.  
\end{align}
\end{theorem}

\noindent {\sc Remarks}: (i) The factor $(2\pi)^n$ that appears dividing $A_n^{\pm}$ in \eqref{20250912_13:18} is a normalization factor due to the fact that we are choosing to work with $1$-periodic functions instead of $2\pi$-periodic functions. Had we defined, for instance, $\widetilde{S_0}(\tau, \chi_{D}):= S_0\big(\frac{\tau}{2\pi}, \chi_{D}\big)$, for $\tau \in \R/2\pi\Z$, its antiderivatives would be $\widetilde{S_n}(\tau, \chi_{D}):= (2\pi)^n S_n\big(\frac{\tau}{2\pi}, \chi_{D}\big)$, and the factor $(2\pi)^n$ would be removed from \eqref{20250912_13:18} for $\widetilde{S_n}(\tau, \chi_{D})$.

\smallskip

\noindent (ii) In light of Lemma \ref{Lem:Bernoulli_rep} and the fact that $\mathcal{B}_{n+1}' = (n+1)\mathcal{B}_{n}$ for $n \geq 1$, one observes that 
\begin{align}\label{20250916_12:01}
c_n = - \frac{1}{(n+1)!} \sum_{j=1}^{2g} \mathcal{B}_{n+1}(-\theta_j).
\end{align}
Note that, when $n$ is even, one has $c_n = 0$ by the symmetry of the $\theta_j$'s (note the fact that $\mathcal{B}_{n+1}(- \theta) = (-1)^{n+1} \mathcal{B}_{n+1}(\theta) $; in particular, $\mathcal{B}_{n+1}(0) = \mathcal{B}_{n+1}(\tfrac12) = 0$ when $n$ is even).

\smallskip

\noindent(iii) As we shall see in the proof, in each of the Theorems \ref{Thm1}, \ref{Thm2} and \ref{Thm3} the term $o(1)$ is in fact $O(\log_q \log_q d \, / \log_q d)$ (with implicit constant depending on $n$ in the case of Theorem \ref{Thm3}). 

\smallskip

\noindent (iv) One can consider even degrees $d = 2g+2$ in the definition \eqref{20250908_16:48} as well. In this case, the corresponding $\mathcal{L}(u,\chi_{D})$ in \eqref{20250908_16:49} is a polynomial of degree $2g+1$, with one trivial zero at $u=1$, and the remaining $2g$ zeros verifying the Riemann hypothesis, i.e., on the circle $|u| = q^{-1/2}$. The results in Theorems \ref{Thm1}, \ref{Thm2} and \ref{Thm3} continue to hold in this case, with minor modifications in the proofs to accommodate an additional contribution $O(1)$ from the trivial zero at $u=1$. To avoid such unnecessary technicalities, we opted for the exposition in the case $d=2g +1$, which is standard for this model in the literature; see, e.g., \cite{AGK, BF, BF2, Flo1}.

\subsection{Lessons from the function field setup} \label{Sec1.3} There are a couple of hidden insights and mysterious connections coming from Theorems \ref{Thm1}, \ref{Thm2} and \ref{Thm3} that we now describe.

\subsubsection{Matching values} \label{Sec1.3.1}

We elaborate on the comparison between Theorem \ref{Thm3} and inequality \eqref{20250903_14:30}. It is well-known that, for $n = 4k+1$, 
\begin{align*}
M_{n+1} = B_{n+1}(0) = \frac{2 \cdot (n+1)! \cdot  \zeta(n+1)}{(2\pi)^{n+1}} \ \  ; \ \ m_{n+1} = B_{n+1}(\tfrac12) = - \frac{2 \cdot (n+1)!\, \big(1 - 2^{-n}\big)\,\zeta(n+1)}{(2\pi)^{n+1}} \,,
\end{align*}
and, for $n = 4k-1$,
\begin{align*}
M_{n+1} = B_{n+1}(\tfrac12) = \frac{2 \cdot (n+1)!\, \big(1 - 2^{-n}\big)\,\zeta(n+1)}{(2\pi)^{n+1}} \ \ \ ;  \ \ \ m_{n+1} = B_{n+1}(0) = -  \frac{2 \cdot (n+1)! \cdot  \zeta(n+1)}{(2\pi)^{n+1}}\,;
\end{align*}
see, e.g., \cite[Theorem 1]{Leh}. This plainly implies that, for $n$ odd, we have the exact match
$$A_n^{-} = C_n^{-}  \ \ \ {\rm and} \ \ \ A_n^{+} = C_n^{+}\,,$$
with $C_n^{+}$ and $C_n^-$ defined in \eqref{20250911_17:58}.

\smallskip

On the other hand, for $n\geq 2$ even, Lehmer \cite[Eqs. (15)-(19)]{Leh} showed that $M_{n+1} = - m_{n+1}$ and 
\begin{align*}
 \frac{2 \cdot (n+1)! \, ( 1 - 3^{-n})}{(2\pi)^{n+1}}  < M_{n+1} <  \frac{2 \cdot (n+1)!}{(2\pi)^{n+1}} .
\end{align*}
This implies that, for $n\geq 2$ even, 
\begin{align}\label{20250911_23:58}
 \frac{(1 - 3^{-n})}{\pi \cdot 2^{n+1}} < A_n^- = A_n^+ < \frac{1}{\pi \cdot 2^{n+1}}.
\end{align}
Inequality \eqref{20250911_23:58} in turn implies that, for $n \geq 2$ even,  
\begin{align}\label{20250912_00:13}
A_n^{\pm} < C_n^{\pm}\,,
\end{align}
with $C_n^{\pm}$ defined in \eqref{20250911_23:59}. This is a simple numerical verification for small $n$ as $C_n^{\pm} \sim \frac{\sqrt{2}}{\pi \cdot 2^{n+1}}$ quickly. From \eqref{20250912_00:13} we see that we do have a refined analogue in these cases, and that we are actually winning by a factor of $\sqrt{2}$ asymptotically.

\smallskip

The analogy between the function field setup and the number field setup motivates the following questions in the theory of the Riemann zeta-function.

\makeatletter
\renewcommand{\thetheorem}{\Alph{theorem}} 
\setcounter{theorem}{0}
\makeatother

\begin{question}[Exact match]\label{QuestionA}
Let $A_n^{\pm}$ be defined as in \eqref{20250912_00:19}. For $n\geq 2$ even, as $t \to \infty$, do we have 
\begin{equation*}
-\left( A^-_n + o(1)\right) \frac{\log t}{(\log \log t)^{n+1}} \ \leq \ S_n(t) \ \leq \ \left( A^+_n + o(1)\right) \frac{\log t}{(\log \log t)^{n+1}}\,?
\end{equation*}
\end{question}

One can weaken a bit the formulation and ask for an approximate match instead. From \eqref{20250911_23:58} note that $A_n^{\pm} \sim \frac{1}{\pi \cdot 2^{n+1}}$ as $n \to \infty$.

\begin{question}[Approximate match]\label{QuestionB}
For $n\geq 2$ even, as $t \to \infty$, do we have
\begin{equation*}
-\left( \widetilde{A_n^{-}} + o(1)\right) \frac{\log t}{(\log \log t)^{n+1}} \ \leq \ S_n(t) \ \leq \ \left( \widetilde{A^+_n} + o(1)\right) \frac{\log t}{(\log \log t)^{n+1}}\,,
\end{equation*}
for certain constants $ \widetilde{A_n^{\pm}}$ verifying $\lim_{n \to \infty} \pi \cdot 2^{n+1}\cdot \widetilde{A_n^{\pm}}   = 1$\,?
\end{question}

\makeatletter
\renewcommand{\thetheorem}{\arabic{theorem}} 
\setcounter{theorem}{3}
\makeatother

\subsubsection{Strategy and mysterious connections} A classical problem in approximation theory is the problem of finding trigonometric polynomials of a given degree that majorize/minorize a given periodic function, optimizing the $L^1(\R/\Z)$-error. This problem has been treated, for instance, in \cite{Car1, CLV, CV, Mon, Vaa}. It can be seen as the periodic analogue of the classical Beurling--Selberg extremal problem of finding one-sided bandlimited approximations for a given real-valued function on $\R$, optimizing the $L^1(\R)$-error; see, e.g., \cite{CL, CLV, CV, Lit, Vaa}. These two worlds meet again in this paper, but they do so in a somewhat mysterious way.

\smallskip

One of the key messages from this paper is that the proofs of Theorems  \ref{Thm1}, \ref{Thm2} and \ref{Thm3} are amenable to a common general strategy: (i) finding a representation formula for our objects in terms of sums of translates of a suitable periodic function;  (ii) finding extremal trigonometric polynomials of degree $N$ that majorize/minorize such a periodic function; (iii) estimating the sum over zeros with the prime polynomial theorem in order to choose the degree $N$ optimally. The periodic function that arises in connection to Theorem \ref{Thm1} is $\varphi(\theta) = \log 2| \sin(\pi\theta)|$, while the one in connection to $S_n(\theta, \chi_D)$ in Theorems \ref{Thm2} and \ref{Thm3} is the Bernoulli periodic function $\mathcal{B}_{n+1}(\theta)$. It turns out that the extremal trigonometric polynomials that majorize/minorize these functions have already been characterized in \cite{Car1,CV}.

\smallskip

The proofs of \eqref{20250903_14:28}, \eqref{20250903_14:29}, and \eqref{20250903_14:30} when $n$ is odd, follow a parallel strategy via the Guinand-Weil explicit formula applied to the Beurling-Selberg extremal bandlimited majorants and minorants of certain special real-valued functions in each case \cite{CCM, CC, CS}. These are represented in the table below. 
\vspace{-0.3cm}
\renewcommand{\arraystretch}{1.7}
\begin{align*}
\begin{tabular}{|c|c|c|}
\hline
 &  periodic setting &  real line setting     \\[0.1cm] 
 \hline
 modulus&  $\log 2| \sin(\pi\theta)|$&  $ \log \left(\displaystyle\frac{4+x^2}{x^2}\right)$    \\[0.2cm]
 \hline
  $S_0$&  $\mathcal{B}_{1}(\theta)$& $ \arctan\left(\displaystyle\frac{1}{x}\right) - \displaystyle\frac{x}{1+x^2}$     \\[0.2cm]
   \hline 
 $S_{2k+1}$&  $\mathcal{B}_{2k+2}(\theta)$&  $ \dfrac{1}{(2k+1)}\left[(-1)^{k+1}x^{2k+1}\arctan\left(\frac{1}{x}\right) + \sum_{\ell=0}^{k}\dfrac{(-1)^{k-\ell}}{2\ell+1}x^{2k-2\ell}\right]$  
     \\[0.2cm]
 \hline
\end{tabular}
\end{align*}

\smallskip

\noindent For each of the functions on the right column above, the solution of the Beurling-Selberg extremal problem was achieved \cite{CL, CLV, CV}, and that is what allows the overall strategy to work. In principle, the function counterparts in the table above bear no obvious relation to each other (e.g., the ones on the left are not the periodizations of the ones on the right) and yet, when the solutions of both extremal problems for such functions are put into our framework, after a suitable cancellation of other components of the main term arising from the explicit formulas, one is mysteriously led to the matching constants $A_{n}^{\pm} = C_{n}^{\pm}$ for $n=0$ and $n$ odd (and the $(\log 2)/2$ in the case of the modulus).

\smallskip

In the case of $n = 2k$ with $k\geq 1$, as pointed out in \cite{CC}, there is a function $f_{2k}$ naturally connected to the problem of bounding $S_{2k}(t)$, namely 
\begin{equation}\label{20250922_15:25}
f_{2k}(x)=(-1)^k x^{2k}\arctan\left(\frac{1}{x}\right) +\sum_{\ell=0}^{k-1}\frac{(-1)^{k-\ell+1}}{2\ell+1} \,x^{2k-2\ell-1} -\dfrac{x}{(2k+1)(1+x^2)}.
\end{equation}  
The main obstacle in this case is that the Beurling-Selberg extremal problem for these functions is not yet solved. These are odd and continuous functions, a class that is considered particularly difficult for this problem. The authors in \cite{CC} then  develop an alternative (non-optimal) approach in order to arrive at \eqref{20250903_14:30} in this case. Our wishful thinking is that, if one were actually able to solve the Beurling-Selberg extremal problem for $f_{2k}$ in \eqref{20250922_15:25} and run the complete strategy, one would likely be in a position to answer Questions \ref{QuestionA} and \ref{QuestionB} above.

\section{Auxiliary results}

\subsection{The prime polynomial theorem} The zeta-function over $\mathbb{A} = \F_q[x]$ is defined by
$$\zeta_{\mathbb{A} }(s) := \sum_{f \ {\rm monic}} \frac{1}{|f|^{s}}=\prod_{P \ {\rm prime}}\left(1-|P|^{-s} \right)^{-1},\;\; \operatorname{Re}(s)>1.$$
With the usual change of variables $u = q^{-s}$, and observing that there are $q^k$ monic polynomials of degree $k$, one arrives at 
\begin{align}\label{20250910_13:02}
\mathcal{Z}_{\mathbb{A} }(u) = \sum_{f\ {\rm monic}} u^{\deg(f)}=\prod_{P\ {\rm prime}}\left(1- u^{\deg(P)} \right)^{-1} = \frac{1}{1 - qu}.
\end{align}
Logarithmic differentiation of \eqref{20250910_13:02} then yields the prime polynomial theorem
\begin{align}\label{20250910_13:31}
\sum_{\substack{f \ {\rm monic} \\ \deg(f) = k}} \Lambda(f) = q^k,
\end{align}
where, for a monic polynomial $f$, we write $\Lambda(f) := \deg(P)$ if $f = P^k$ for some prime polynomial $P$ and positive integer $k$, and $\Lambda(f) := 0$ otherwise.

\smallskip

In a similar way, logarithmic differentiation of \eqref{20250908_16:49} and \eqref{20250908_17:40} yields
\begin{align}\label{20250910_13:27}
\sum_{k=1}^\infty \left( \sum_{\substack{f \ {\rm monic} \\ \deg(f) = k}}  \chi_{D}(f)\, \Lambda(f)\right) u^{k-1} = \frac{\mathcal{L}'}{\mathcal{L}}(u,\chi_{D}) = - \sum_{k=1}^\infty \left( \sum_{j=1}^{2g} u_j^{-k}\right) u^{k-1}.
\end{align}
Recalling that $u_j = q^{-1/2} e(\theta_j)$, expression \eqref{20250910_13:27} leads to the following identity, for each $k \geq 1$, 
\begin{align}\label{20250910_13:30}
- \sum_{j=1}^{2g} e(-k \theta_j) = \frac{1}{q^{k/2}} \sum_{\substack{f \ {\rm monic} \\ \deg(f) = k}}  \chi_{D}(f) \,\Lambda(f).
\end{align}

\smallskip

By combining \eqref{20250910_13:31} and \eqref{20250910_13:30} we get, for each $k \neq 0$, 
\begin{align}\label{20251910_14:09}
\left|\sum_{j=1}^{2g} e(k \theta_j)\right| = \left|\frac{1}{q^{|k|/2}} \sum_{\substack{f \ {\rm monic} \\ \deg(f) = |k|}}  \chi_{D}(f) \,\Lambda(f) \right| \leq \frac{1}{q^{|k|/2}} \sum_{\substack{f \ {\rm monic} \\ \deg(f) = |k|}}  \Lambda(f)  = q^{|k|/2}.
\end{align}
Estimate \eqref{20251910_14:09} will be later used in the proofs of Theorems \ref{Thm1}, \ref{Thm2} and \ref{Thm3}.

\subsection{Representation formulas} We start here with a simple observation.
\begin{lemma}\label{Lem4_modulus}
For each $\theta \in \R/\Z$, we have
\begin{align*}
\log\big|\mathcal{L}\big(q^{-1/2} e(\theta),\chi_{D}\big)\big| = \sum_{j=1}^{2g} \log 2| \sin\pi(\theta - \theta_j)|.
\end{align*}
\end{lemma}
\begin{proof}
This plainly follows from \eqref{20250908_17:40}. 
\end{proof}

Recall that, for $n\geq 1$, the Bernoulli periodic functions have the Fourier series expansion 
\begin{equation}\label{20250909_14:44}
\mathcal{B}_n(\theta) = -\dfrac{n!}{(2\pi i)^n} \sum_{\substack{k=-\infty \\ k \neq 0}}^{\infty}\dfrac{1}{k^n} \, e(k \theta)
\end{equation}
(for $n=1$ the equality above is understood as the limit of the symmetric partial sums). In particular, for $n\geq1$, note that each $\mathcal{B}_n$ has mean value zero in $\R/\Z$, i.e., $\widehat{\mathcal{B}_n}(0) = 0$. 

\smallskip

We proceed with a formula that connects the functions $S_n(\theta, \chi_{D})$, for $n\geq 0$, to the Bernoulli periodic functions. The case $n=0$ of the following lemma already appears in \cite[Theorem 4]{AGK}, and we include a brief proof for the convenience of the reader. 

\begin{lemma}\label{Lem:Bernoulli_rep}
Let $n$ be a non-negative integer. For each $\theta \in \R/\Z$, we have
\begin{align}\label{20250905_16:46}
S_n(\theta, \chi_{D}) =  - \frac{1}{(n+1)!}\sum_{j=1}^{2g} \mathcal{B}_{n+1}(\theta - \theta_j).
\end{align}
\end{lemma}
\begin{proof} We first prove the case $n=0$. Assume first that $\theta \neq \theta_j$. Considering the principal branch of the logarithm, from \eqref{20250908_17:40} and definition \eqref{20250911_14:12} we get 
\begin{align*}
S(\theta, \chi_{D}) & =  - \frac{1}{\pi} \, \Delta_{\Gamma_{\theta}}\arg \mathcal{L}(u,\chi_{D}) = - \frac{1}{\pi} \sum_{j=1}^{2g} \Delta_{\Gamma_{\theta}}\arg  \left(1 -  \frac{u}{u_j}\right)\\
 &  = - \frac{1}{\pi} \sum_{j=1}^{2g}  \operatorname{Im} \left( \log  \left(1 -  e(\theta - \theta_j) \right) - \log  \left(1 -  q^{-3/2} e(-\theta_j)\right)\right).
\end{align*} 
Since the zeros $u_j = q^{-1/2} e(\theta_j)$ appear in conjugate pairs, we have 
\begin{align*}
\sum_{j=1}^{2g} \operatorname{Im}\log  \left(1 -  q^{-3/2} e(-\theta_j)\right) = 0.
\end{align*}
Noting that
\begin{align*}
\operatorname{Im} \log  \left(1 -  e(\theta) \right) = \pi  \, \mathcal{B}_{1}(\theta)
\end{align*}
for any $\theta \in \R / \Z$, we arrive at the desired identity \eqref{20250905_16:46} in the case $n=0$.

\smallskip

The general case plainly follows by induction from the way we defined $S_n(\theta, \chi_{D})$ in \eqref{20250908_17:29}, using the fact that the right-hand side of \eqref{20250905_16:46} has mean value zero in $\R/\Z$, and $\mathcal{B}_{n+1}'(x) = (n+1)\mathcal{B}_{n}(x)$ for $n \geq 1$ (for $n=1$ it suffices that this holds on $\R/\Z \setminus \{0\}$).  
\end{proof}

\subsection{Extremal trigonometric polynomials} 
\smallskip

We collect here a few results on one-sided approximations of periodic functions by trigonometric polynomials that will be suitable for our purposes. The first one is a result from Carneiro and Vaaler \cite[Theorem 1.5]{CV}.

\begin{lemma} \label{Lem6_maj_varphi}
Let $\varphi:\R /\Z \to \R$ be given by $\varphi(\theta) = \log 2| \sin(\pi\theta)|$, and let $N$ be a non-negative integer.  Then there exists a unique real-valued trigonometric polynomial $U_N(\theta) = \sum_{k = -N}^{N}  \widehat{U_N}(k) \, e(k\theta) $
such that $\varphi(\theta) \leq U_N(\theta) $ for each $\theta \in \R/\Z$ and such that 
\begin{align*}
\int_{\R/\Z} U_N(\theta) \, \d \theta= \frac{\log 2}{N+1}. 
\end{align*}
Moreover, for each integer $k$ with $1 \leq |k| \leq N$, one has 
\begin{align*}
-\frac{1}{2|k|} \leq \widehat{U_N}(k)  \leq 0.
\end{align*}
\end{lemma}

The next lemma is a result from \cite[Theorem 1]{Car1}, characterizing the extremal approximations to the Bernoulli periodic functions. It is worth mentioning that these, in turn, rely on the solution of the Beurling-Selberg extremal problem for the functions $x^k \,\sgn(x)$, for $k\geq 1$, due to Littmann \cite{Lit} (the case $k=0$ is originally due to Beurling, as remarked in Vaaler's classical survey \cite{Vaa}). For $n \geq 0$, recall the definition of $M_{n}$ and $m_n$ in \eqref{20250912_17:34}.

\begin{lemma}\label{Lem7}
Let $n$ and $N$ be non-negative integers. There exist unique trigonometric polynomials  $P_{n+1,N}^{\pm}(\theta) =\sum_{k = -N}^{N}  \widehat{P_{n+1,N}^{\pm}}(k) \, e(k\theta)$ such that 
\begin{align*}
P_{n+1,N}^{-}(\theta) \leq \mathcal{B}_{n+1}(\theta) \leq P_{n+1,N}^{+}(\theta)
\end{align*}
for each $\theta \in \R / \Z$ and 
\begin{align*}
\int_{\R/\Z} P_{n+1,N}^{-}(\theta)  \, \d \theta = \frac{m_{n+1}}{(N+1)^{n+1}}  \ \ \ \ {\rm and} \ \ \ \ \int_{\R/\Z} P_{n+1,N}^{+}(\theta)  \, \d \theta = \frac{M_{n+1}}{(N+1)^{n+1}}.
\end{align*}
Moreover, for each integer $k$ with $1 \leq |k| \leq N$, one has 
\begin{align}\label{20250909_14:34}
\left| \widehat{P_{n+1,N}^{\pm}}(k) \right| \ll_n \frac{1}{k^{n+1}}.
\end{align}
\end{lemma}

\begin{proof}[Proof of \eqref{20250909_14:34}] The estimate \eqref{20250909_14:34} is not explicitly mentioned in \cite[Theorem 1]{Car1} and we give a brief proof. We deal with the case of $P_{n+1,N}^{+}$, and the estimate for $P_{n+1,N}^{-}$ is analogous. Let $D_{n+1, N}^+(\theta) := P_{n+1,N}^{+}(\theta) - \mathcal{B}_{n+1}(\theta)$. Since $D_{n+1, N}^+(\theta) \geq 0$, we have 
\begin{align}\label{20250909_14:43}
\left| \widehat{D_{n+1,N}^+}(k) \right| \leq \widehat{D_{n+1,N}^{+}}(0) =  \widehat{P_{n+1,N}^{+}}(0)  = \frac{M_{n+1}}{(N+1)^{n+1}}.
\end{align} 
By the triangle inequality,
\begin{align*}
\left| \widehat{P_{n+1,N}^+}(k) \right| \leq \left| \widehat{D_{n+1,N}^+}(k) \right| + \left|\widehat{\mathcal{B}_{n+1}}(k)\right|,
\end{align*}
and estimate \eqref{20250909_14:34} plainly follows from \eqref{20250909_14:43} and \eqref{20250909_14:44}, since $\widehat{\mathcal{B}_{n+1}}(k) = - \frac{(n+1)!}{(2 \pi i k)^{n+1}}$.\end{proof}

The next lemma is classical; see, e.g., \cite{Mon}. We include a brief outline of the proof. If  $I = [\alpha, \beta]$ is an interval on $\R / \Z$ with length $\beta - \alpha \leq 1$, we denote by $\mathbf{1}_I$ be the normalized characteristic function of $I$ (i.e., with $\mathbf{1}_I(\alpha) = \mathbf{1}_I(\beta) = \tfrac12$ if $0 < \beta - \alpha < 1$). 

\begin{lemma} \label{lem_Mont_char_intervals}
Let $I = [\alpha, \beta]$ be an interval on $\R / \Z$ with length $\beta - \alpha \leq 1$, and let $N$ be a non-negative integer. Then there exist real-valued trigonometric polynomials $T_N^{\pm}(\theta) = \sum_{k = -N}^{N}  \widehat{T_N^{\pm}}(k) \, e(k\theta) $ such that $T_N^{-}(\theta) \leq \mathbf{1}_I(\theta) \leq T_N^{+}(\theta)$ for each $\theta \in \R / \Z$, and 
\begin{align}\label{20250911_14:21}
\int_{\R/ \Z} \big(T_N^{+}(\theta) - \mathbf{1}_I(\theta)\big) \, \d\theta = \int_{\R/ \Z} \big( \mathbf{1}_I(\theta) - T_N^{-}(\theta)\big) \, \d\theta = \frac{1}{N+1}.
\end{align}
Moreover, for each integer $k$ with $1 \leq |k| \leq N$, one has $\Big| \widehat{T_N^{\pm}}(k) \Big| \ll1$.
\end{lemma}
\begin{proof} This is a consequence of the case $n=0$ of Lemma \ref{Lem7}. One observes that $\mathbf{1}_I(\theta) = (\beta - \alpha) + \mathcal{B}_1(\alpha - \theta) + \mathcal{B}_1(\theta - \beta)$ and then takes $T_N^{\pm}(\theta) = (\beta - \alpha) + P_{1,N}^{\pm}(\alpha - \theta) + P_{1,N}^{\pm}(\theta - \beta)$.
\end{proof}

\subsection{Zero counting function for $\mathcal{L}(u,\chi_{D})$} Let $I = [\alpha, \beta]$ be an interval on $\R / \Z$ with length $\beta - \alpha \leq 1$. Denote by $N(I, \chi_{D})$ the (normalized) number of zeros $q^{-1/2}e(\theta_j)$ of $\mathcal{L}(u,\chi_{D})$ with $\theta_j \in I$, that is
\begin{align*}
N(I, \chi_{D}) = \sum_{j=1}^{2g}\mathbf{1}_I(\theta_j).
\end{align*} 
Small variants of the next lemma have appeared in \cite{AGK, FR}. We provide a brief proof for the convenience of the reader.
\begin{lemma}\label{Lem9_zero}
Let $I = [\alpha, \beta]$ be an interval on $\R / \Z$ with length $\beta - \alpha \leq 1$. Then 
\begin{align*}
N(I, \chi_{D}) = 2g\cdot(\beta - \alpha) + S(\beta, \chi_{D}) - S(\alpha, \chi_{D}).
\end{align*}
\end{lemma}
\begin{proof}
Assume first that $q^{-1/2}e(\alpha)$ and $q^{-1/2}e(\beta)$ are not zeros of $\mathcal{L}(u , \chi_D)$. Consider the positively oriented contour $\Gamma = \Gamma_1 \cup \Gamma_2$, consisting only of circular arcs and rays, where $\Gamma_1$ connects the points $q^{-1/2}e(\alpha)$, $q \,e(\alpha)$, $q \, e(\beta)$ and $q^{-1/2}e(\beta)$, and $\Gamma_2$ connects the points $q^{-1/2}e(\beta)$, $q^{-2}e(\beta)$, $q^{-2}e(\alpha)$ and $q^{-1/2}e(\alpha)$. By the argument principle note that 
\begin{align}\label{20250911_14:10}
N(I, \chi_{D}) = \frac{1}{2\pi} \Delta_{\Gamma} \arg\mathcal{L}(u,\chi_{D}).
\end{align}
By the functional equation \eqref{20250910_15:16} and the fact that $\overline{\mathcal{L}(\overline{u},\chi_{D})} = \mathcal{L}(u,\chi_{D})$, one has 
\begin{align}\label{20250911_14:11}
\frac{1}{2\pi} \Delta_{\Gamma_1} \arg\mathcal{L}(u,\chi_{D}) = 2g\cdot (\beta - \alpha) + \frac{1}{2\pi} \Delta_{\Gamma_2} \arg\mathcal{L}(u,\chi_{D}).
\end{align}
From \eqref{20250911_14:10}, \eqref{20250911_14:11}, and definition \eqref{20250911_14:12}, we arrive at 
\begin{align*}
N(I, \chi_{D})  = 2g\cdot (\beta - \alpha)  + \frac{1}{\pi} \Delta_{\Gamma_2} \arg\mathcal{L}(u,\chi_{D}) = 2g\cdot (\beta - \alpha) + S(\beta, \chi_{D}) - S(\alpha, \chi_{D}).
\end{align*}
If  $q^{-1/2}e(\alpha)$ and/or $q^{-1/2}e(\beta)$ are zeros of $\mathcal{L}(u , \chi_D)$, the result follows from the normalized way we defined $N(I, \chi_{D})$ and $S(\theta, \chi_{D})$.
\end{proof}

\section{Proofs of Theorems \ref{Thm1}, \ref{Thm2} and \ref{Thm3}}
We give a unified proof of Theorems \ref{Thm1}, \ref{Thm2} and \ref{Thm3}. In each situation we want to bound a certain periodic function $F (\cdot, \chi_D): \R/\Z \to \R$ that can be written as
\begin{equation*}
F(\theta, \chi_D) = \sum_{j=1}^{2g} G(\theta - \theta_j)\,,
\end{equation*}
with $G:\R/\Z \to \R$ being a suitable periodic function of mean value zero (these are Lemmas \ref{Lem4_modulus} and \ref{Lem:Bernoulli_rep}). Suppose we can solve the problem of majorizing $G$ by a trigonometric polynomial of a fixed degree, optimizing the $L^1(\R/\Z)$-error. In our framework (Lemmas \ref{Lem6_maj_varphi} and \ref{Lem7}), this means that, for each non-negative $N$, there exists a trigonometric polynomial $V_N(\theta) = \sum_{k=-N}^{N} \widehat{V_N}(k) \,e(k\theta)$ such that:
\begin{itemize}
\item[(i)] $G(\theta) \leq V_N(\theta)$ for all $\theta \in \R/\Z$. 
\item[(ii)] For some constants ${\bf C}$ and $n$ (independent of $N$), we have
$$\int_{\R/\Z} V_N(\theta) \,\d\theta = \frac{{\bf C}}{(N+1)^{n+1}}.$$
\item[(iii)] $\big|\widehat{V_N}(k)\big| \ll 1$ for integers $k$ with $1 \leq |k| \leq N$, with the implicit constant independent of $N$. 
\end{itemize}
Under such conditions, we have
\begin{align*}
F(\theta, \chi_D) & = \sum_{j=1}^{2g} G(\theta - \theta_j) \leq  \sum_{j=1}^{2g} V_N(\theta - \theta_j)  = \sum_{j=1}^{2g} \sum_{k=-N}^{N} \widehat{V_N}(k) \,e\big(k(\theta - \theta_j)\big)\\
& = 2g \,\widehat{V_N}(0) + \sum_{j=1}^{2g} \sum_{\substack{k=-N \\ k\neq0}}^{N} \widehat{V_N}(k) \,e\big(k(\theta - \theta_j)\big) \leq 2g \,\widehat{V_N}(0) + \sum_{\substack{k=-N \\ k\neq0}}^{N} \Big|\widehat{V_N}(k) \Big| \left| \sum_{j=1}^{2g} e(k\theta_j)\right|\\
& \leq 2g \,\widehat{V_N}(0) +  \sum_{\substack{k=-N \\ k\neq0}}^{N} \Big|\widehat{V_N}(k) \Big| \, q^{|k|/2} = \frac{2g\,{\bf C}}{(N+1)^{n+1}} + O\big(N q^{N/2}\big) \,,
\end{align*}
where we have used \eqref{20251910_14:09}. Choosing $N \sim 2\log_q d - (2n+6)\log_q \log_q d$, we get
\begin{align*}
F(\theta, \chi_D) \leq  \left( \frac{{\bf C}}{2^{n+1}} + o(1) \right) \frac{d}{(\log_q d)^{n+1}}.
\end{align*}
The term $o(1)$ above is $O(\log_q \log_q d \, / \log_q d )$ (the implicit constant might depend on $n$ here). 

\smallskip

The proof of the lower bound is analogous, using the trigonometric polynomials that minorize $G$. 

\section{An alternative proof of Theorem \ref{Thm2}}
This proof uses the zero counting function and trigonometric polynomials that majorize/minorize the characteristic function of an interval in the periodic setting (Lemma \ref{lem_Mont_char_intervals}). The main insight in this proof, to refine \eqref{20250911_14:20} by a factor of $2$, is the same insight that allows the proof of \eqref{20250903_14:29} for $S(t)$ given by Carneiro, Chandee and Milinovich in \cite{CCM2} to refine a previous estimate of Goldston and Gonek \cite{GG} by a factor of $2$, namely, the use of a symmetric interval to capture the symmetry of the zeros allied to the fact that the error of the approximations in \eqref{20250911_14:21} depend only on the degree of the trigonometric polynomial but not on the size of the interval. 

\begin{proof}[Proof of Theorem \ref{Thm2}] Since $S(\theta, \chi_D) = - S(-\theta, \chi_D)$, we may assume without loss of generality that $0 \leq \theta \leq \frac12$. Let us do the proof of the upper bound with the majorant, as the proof of the lower bound with the minorant is analogous. We use Lemmas \ref{lem_Mont_char_intervals} and \ref{Lem9_zero} with the interval $I = [-\theta, \theta]$ to get
\begin{align*}
2\, S(\theta, \chi_D) & =  - 4g \theta + N(I, \chi_{D}) =  - 4g \theta+ \sum_{j=1}^{2g} \mathbf{1}_I(\theta_j) \leq - 4g \theta + \sum_{j=1}^{2g} T_N^{+}(\theta_j)\\
& = - 4g \theta + \sum_{j=1}^{2g} \sum_{k=-N}^{N} \widehat{T_N^+}(k) \,e(k\theta_j) =  \left(- 4g \theta + 2g \,\widehat{T_N^+}(0)\right)+  \sum_{j=1}^{2g} \sum_{\substack{k=-N \\ k \neq 0}}^{N} \widehat{T_N^+}(k) \,e(k\theta_j)\\
& = \frac{2g}{N+1} +  \sum_{j=1}^{2g} \sum_{\substack{k=-N \\ k \neq 0}}^{N} \widehat{T_N^+}(k) \,e(k\theta_j) \leq \frac{2g}{N+1}  + \sum_{\substack{k=-N \\ k \neq 0}}^{N}  \left|\widehat{T_N^+}(k) \right| \left| \sum_{j=1}^{2g} e(k\theta_j)\right|\\
& \leq  \frac{2g}{N+1}  + \sum_{\substack{k=-N \\ k \neq 0}}^{N}  \left|\widehat{T_N^+}(k) \right| q^{|k|/2} =  \frac{2g}{N+1} + O\big(N q^{N/2}\big)\,,
\end{align*}
where we have used \eqref{20251910_14:09}. Choosing $N \sim 2\log_q d - 6\log_q \log_q d$, we get
\begin{align*}
S(\theta, \chi_D) \leq \left(\frac{1}{4} + o(1) \right) \frac{d}{\log_q d}.
\end{align*}

\end{proof}

\section*{Acknowledgments}
The second author gratefully acknowledges the hospitality of ICTP, Trieste, where this work was initiated, the financial support of the Australian Research Council Grant DP230100534, and the Max Planck Institute for Mathematics, Bonn. The third author would like to thank the DST - Government of India for the support under the DST-INSPIRE Faculty Scheme with Faculty Reg. No. IFA21-MA 168.

\end{document}